\documentclass[reqno]{amsart}
\usepackage{amsmath}
\usepackage{amssymb}
\usepackage{paralist}
\usepackage{graphics} %% add this and next lines if pictures should be in esp format
\usepackage[colorlinks=true]{hyperref}
\usepackage{mathrsfs}
\usepackage{extarrows}
\usepackage{enumerate}
\usepackage{float}

\newtheorem{theorem}{Theorem}[section]

\newtheorem{lemma}[theorem]{Lemma}
\newtheorem{proposition}[theorem]{Proposition}

\theoremstyle{definition}
\newtheorem{definition}[theorem]{Definition}
\newtheorem*{remark}{Remark}

\newtheorem{example}[theorem]{Example}
\newcommand{\ep}{\varepsilon}

\newcommand{\RR}{\mathbb{R}}

\newcommand{\NN}{\mathbb{N}}

\newcommand{\ZZ}{\mathbb{Z}}

\newcommand{\cF}{\mathcal{F}}

\newcommand{\cN}{\mathcal{N}}
\newcommand{\cm}{\mathcal{M}}

\newcommand{\sss}{\mathscr{S}}
\newcommand{\sC}{\mathscr{C}}

\newcommand{\sg}{\mathscr{G}}

\newcommand{\al}{\alpha}

\title[Unique Ergodicity and the Approximate Product Property]
{Unique ergodicity for zero-entropy dynamical systems with the approximate product property}

\author[Peng Sun]{}


\email{sunpeng@cufe.edu.cn}

\usepackage{tikz}
\usetikzlibrary{calc,trees,positioning,arrows,chains,shapes.geometric,patterns,%
	decorations.pathreplacing,decorations.pathmorphing,shapes,%
	matrix,shapes.symbols}

\tikzset{
	>=stealth',
	punktchain/.style={
		rectangle, 
		rounded corners, 
		draw=black, very thick,
		text width=9em, 
		minimum height=3em, 
		text centered, 
		on chain},
	line/.style={draw, thick, <-},
	element/.style={
		tape,
		top color=white,
		bottom color=blue!50!black!60!,
		minimum width=5em,
		draw=blue!40!black!90, very thick,
		text width=9em, 
		minimum height=3em, 
		text centered, 
		on chain},
	every join/.style={->, thick,shorten >=1pt},
	decoration={brace},
	tuborg/.style={decorate},
	tubnode/.style={midway, right=2pt},
}

\begin{document}
	
	\maketitle\ 
	
	%\bigskip
	
	% Enter the first author's name and address:
	\centerline{\scshape Peng Sun}
	\medskip
	{\footnotesize
		% please put the address of the first author
		\centerline{China Economics and Management Academy}
		\centerline{Central University of Finance and Economics}
		\centerline{Beijing 100081, China}
	} % Do not forget to end the {\footnotesize by the sign }

	\bigskip

	%+Abstract
	\begin{abstract}
We show that for every topological dynamical system with the approximate product property,
zero topological entropy is equivalent to unique ergodicity.
Equivalence of minimality is also proved under a slightly stronger condition.
Moreover, we show that unique ergodicity implies the approximate product property if the system has periodic points.
	\end{abstract}
	%-Abstract
	
	%+Contents
	%\tableofcontents
	%-Contents

%\baselineskip 15pt

\section{Introduction}

It has been a historic question how the properties of a dynamical system are
determined by its topological entropy. 
Usually we tend to regard that a system
with zero topological entropy is considerably simple, while positive
topological entropy often comes with a rich structure.
Such facts can be shown for systems in certain classes. For example, in the seminal work
\cite{Ka80} Katok showed that for $C^{1+\alpha}$ diffeomorphisms on a surface,
positive topological entropy is equivalent to the existence of horseshoes.
On the contrary, there are zero-entropy systems that are in some sense complicated,
as well as positive-entropy systems that are in some sense simple. 
The question raised by Parry whether strict ergodicity implies
zero topological entropy has a negative answer in the most general setting. 
See Section \ref{seceg} for examples with more details. 

As the theory of hyperbolicity had been developed, Herman expected
a positive answer to Parry's question in the smooth case.
Katok then suggested a more ambitious conjecture that every $C^2$ diffeomorphism has
ergodic measures of arbitrary intermediate entropies, i.e. for each $\al\in[0,h(f))$,
where $h(f)$ denotes the topological entropy of the system $(X,f)$,
there is an ergodic measure $\mu_\al$ %for the system $(X,f)$ 
such that its metric entropy satisfies
$h_{\mu_\al}(f)=\al$. Partial results on Katok's conjecture
have been obtained in 
\cite{GSW, Ka80, KKK, QS, Sun09, Sun10, Sun12, Sun20, Sunze, Sunintent, Sunes, Ures}.
%\cite{Ka80}, \cite{GSW}, \cite{QS}, \cite{Sun09}, \cite{Sun10}, \cite{Sun12}, \cite{Sunintent} and \cite{Ures}.
The 
success of specification-like
properties that played pivotal roles in 
\cite{GSW, QS, Sun20, Sunintent}
%\cite{QS}, \cite{GSW} and \cite{Sunintent} 
urges us to
consider classes of systems with such topological properties, which 
are closely related to some sort of hyperbolicity.

Specification-like properties are weak variations of the specification property
introduced by Bowen \cite{Bowen}. 
Since his pioneering works,
plenty of interesting results
have been obtained through various specification-like properties.
Among these properties, the approximate product property
introduced by Pfister and Sullivan \cite{PfSu} is almost the weakest one.
In \cite{Sunintent}, we have verified
Katok's conjecture for every system $(X,f)$ that
satisfies the approximate product property and asymptotically entropy expansiveness.
In fact, we
obtained a much stronger result, a description of the subtle structure of $\cm(X,f)$, the space of 
invariant measures, concerning their metric entropies. In particular, such a system $(X,f)$ 
must have zero topological entropy if it
is uniquely ergodic.

\begin{theorem}[{\cite[Theorem 1.1]{Sunintent}}]
	\label{alresidual}
	Let $(X,f)$ be an asymptotically entropy expansive system with the
	approximate	product property. Then $(X,f)$ has \emph{the generic structure of
		metric entropies}, i.e.
	for every $\al\in[0, h(f))$,
	$$\cm_e(X,f,\al):=\left\{\mu\in\cm(X,f): \text{ $\mu$ is ergodic and }h_\mu(f)=\alpha\right\}$$
	is a residual (dense $G_\delta$) subset in the compact metric subspace
	$$%\cm^\al=
	\cm^\al(X,f):=\{\mu\in\cm(X,f): h_\mu(f)\ge\alpha\}.$$
\end{theorem}

This paper mainly goes in the opposite direction.
%Suppose that 
Let $(X,f)$ be a system with the approximate product property.
Theorem \ref{alresidual} does not make sense if the system has zero topological
entropy.
It is natural to ask what happens in the zero-entropy case, which also provides conditions for the system to have positive topological entropy.
The entropy denseness property proved in \cite{PfSu} indicates that ergodic measures are dense 
in $\cm(X,f)$.  
However, entropy denseness does not guarantee the existence
of multiple ergodic measures.
We shall show that 
$\cm(X,f)$ is actually a singleton 
if $(X,f)$ has zero topological entropy.
Moreover, we show that the converse is also true even if
asymptotic entropy expansiveness is not assumed.

\begin{theorem}\label{thmunierg}
	Let $(X,f)$ be a system with the approximate product property.
	Then $(X,f)$ is uniquely ergodic if and only if $(X,f)$ has zero
	topological entropy.
\end{theorem}

We remark that the main difficulty we encountered in the proof of Theorem
\ref{thmunierg} comes from the variable gaps 
% in the definition of
caused by the way we define the approximate product
property. Similar difficulties also arise in the works 
\cite{Sun19, Sun18flow, Sunze}
%\cite{Sun19}, \cite{Sun18flow} and \cite{Sunze} 
of the author. 
Each of these works presents a different solution to this issue.
Dealing with the variable gaps is
not a crucial ingredient in other
works on the approximate product property or the gluing orbit property, such as 
\cite{BTV, PfSu, Sunintent}.
%\cite{PfSu}, \cite{BTV},\cite{Sunintent} and so on. 
However, it is just this issue, as well as our results, 
that reflects the
substantial difference between the weaker specification-like properties and
the stronger ones.

Theorem \ref{thmunierg} is more than a positive answer to the question of Parry and
Herman in the class of systems with the approximate product property.
Along with the results in \cite{PfSu} and \cite{Sunintent},
we have a dichotomy 
on the structure of 
$\cm(X,f)$
for a system with
the approximate product property, which is completely determined 
by its topological entropy:
$$\begin{cases}
h(f)=0 \iff {}&{} \text{$\cm(X,f)$ is a singleton}.\\
h(f)>0 \iff {}&{} \text{$\cm(X,f)$ is a Poulsen simplex}.
%\\&\text{$\cm_e(X,f,\al)$ is a residual subset of $\cm^\al(X,f)$ for each }
\end{cases}$$
Moreover, when $h(f)>0$, if in addition $(X,f)$ is asymptotically entropy expansive, then
the system has the generic structure of metric entropies as described in
Theorem \ref{alresidual}.  Readers are referred to \cite{LOS, Phelps} for more details about the Poulsen simplex.

We have shown in \cite{Sunze} that a zero-entropy system with gluing orbit property
must be minimal and equicontinuous.
%Similar results do not 
This does not
hold for systems with
the approximate product property. 
%There are zero-entropy systems with the approximate product property that are topologically
%mixing, as well as such systems that are not topologically transitive. 
A zero-entropy systems with the approximate product property may be topologically mixing or may not be topologically transitive.
See Example \ref{exnonmini}.
%(cf. \cite[Example 9.6 and Example 9.7]{Sunintent}).
We have shown in \cite{Sunintent} that 
both unique ergodicity and zero topological entropy can be derived from
the approximate product product and minimality. 
%(cf. \cite[Corollary 5.11 and Corollary 5.12]{Sunintent}).
It turns out that non-minimality is caused exactly by the mistakes allowed in
the tracing property \eqref{ndeltrace} in the definition of
the approximate product property. Considering this, we introduce a new
notion between the approximate product
property and the tempered gluing orbit property introduced in \cite{Sunintent}.
Under this so-called \emph{strict approximate product property}
(see Definition \ref{defsapp}), minimality is also equivalent to zero topological entropy.
%As a corollary of Theorem \ref{thmunierg}, we have a further
%result for systems with the tempered gluing orbit property introduced in \cite{Sunintent}.

\begin{theorem}\label{thmsapp}
	Let $(X,f)$ be a system with the strict approximate product property.
	Then the following are equivalent:
	\begin{enumerate}
		\item $(X,f)$ is minimal.
		\item $(X,f)$ is uniquely ergodic.
		\item $(X,f)$ has zero topological entropy.
	\end{enumerate}
\end{theorem}

Comparing Theorem \ref{thmsapp} with the result in \cite{Sunze}, we are still not sure  if in the class of systems with
the tempered gluing orbit property or the strict the approximate product property,
zero topological entropy implies uniformly rigidity or equicontinuity.

By examining the examples we find that some phenomena appear in a more general
way, as exhibited in the following theorems.

\begin{theorem}\label{thmper}
	Let $(X,f)$ be a system with a periodic point $p\in X$. Then
	$(X,f)$ has the approximate product property and zero topological entropy 
	if and only if $(X,f)$ is uniquely
	ergodic, i.e. the periodic measure supported on the orbit of $p$ is the unique ergodic measure
	for $(X,f)$.
\end{theorem}

\begin{remark}
	In case that $(X,f)$ has no periodic point, unique ergodicity does not
	imply either the approximate product property or zero topological entropy.
	See Example \ref{exaper}.
\end{remark}

\begin{theorem}\label{zentappinterval}
	Let $f:I\to I$ be a continuous map on a closed interval $I$. Then $(I,f)$ has
	the approximate product property and zero topological entropy if and only if
	it has a unique attracting fixed point, i.e. there is $p\in I$ such that
	$f(p)=p$ and
	$$\lim_{n\to\infty}f^n(x)=p\text{ for every }x\in I.$$
\end{theorem}

Theorem \ref{zentappinterval}
provides a complete description of continuous interval maps with the approximate
product property and zero topological entropy.
By \cite[Corollary 40]{KLO} and \cite[Example 3.1]{BTV}, a continuous graph map has the approximate product property if it is topologically transitive.
As a corollary, 
we obtain another proof of the fact that
%\begin{corollary}
a topologically transitive
continuous interval map (or graph map with at least two periodic orbits)
must have positive topological entropy (cf. \cite[Theorem A]{ADR}).

We remark that by \cite{Blokh} and \cite{Buzzi}, a continuous interval map has the exact specification property 
(see Definition \ref{defpugo})
if and only if it is topologically mixing.
By \cite[Theorem B and C]{ADR},  a totally transitive continuous graph map 
has the periodic specification property
if it has periodic points, otherwise it is topologically conjugate to an irrational rotation and hence has the gluing orbit property.
It is an interesting question 
%natural 
to
ask if there are any other conditions that are related to 
specification-like
properties for graph maps.

Readers are referred to the book \cite{DGS} and the survey \cite{KLO} for an overview of 
the definitions and results of specification-like properties. More discussions on
the gluing orbit property, the tempered gluing orbit property and the approximate product property,
as well as various examples, can be found in \cite{BTV, Sun19, Sunintent}.
%\cite{BTV}, \cite{Sun19} and \cite{Sunintent}.
The relations between various specification-like properties,
in our terminology,
%as in \cite{Sun19, Sunintent}
%\cite{Sun19}, \cite{Sunintent} 
%and this paper,
are summarized in Figure \ref{figrel}.
Analogous relations hold for periodic specification-like properties.

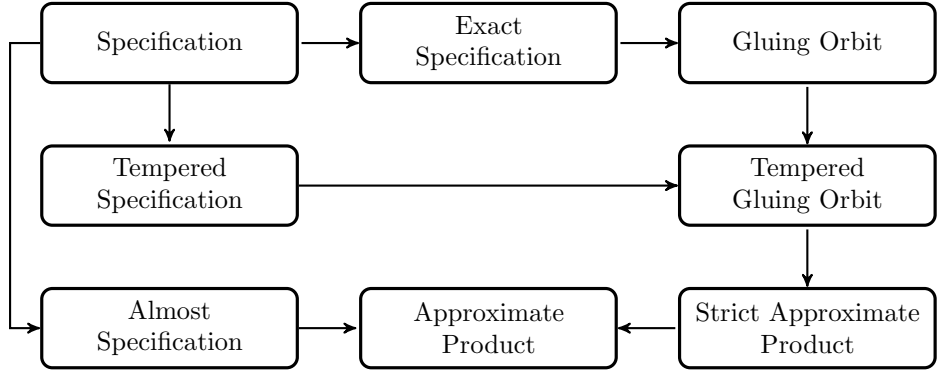
\begin{figure}\label{figrel}
	
	\begin{tikzpicture}%[Figure 1]
	[node distance=.8cm,
	start chain=going below,]

	\node (sp) [punktchain ]  {Specification};
	\begin{scope}[start branch=venstre,
	%We need to redefine the join-style to have the -> turn out right
	every join/.style={->, thick, shorten <=1pt}, ]
	
	\node[punktchain, on chain=going right, join=by {->}] (ugo)  {Exact\\ Specification};
	
	\node[punktchain, on chain=going right, join] (go)      {Gluing Orbit};
	
	\node[punktchain, on chain=going below, join] (tgo) {Tempered Gluing Orbit};
	
	\node[punktchain, on chain=going below, join=by {->}]
	(sapp){Strict Approximate\\ Product};
	
	\node[punktchain, on chain=going left, join=by {->}] (app)      {Approximate Product};
	
	\node[punktchain, on chain=going left, join=by {<-}] (asp)      {Almost\\
		Specification};
	
	\end{scope}
	
	\node[punktchain, on chain=going below, join=by {->}]
	(tsp) {Tempered Specification};

	\draw[->, thick,] (tsp.east)--(tgo.west);

	\draw[->, thick,] (sp.west) --+(-0.4, 0)|- (asp.west);
	
	\end{tikzpicture}
	
	\
	
	\caption{Relations between specification-like properties}
\end{figure}

Notions and results in this paper naturally extends to the continuous-time case, i.e. flows and semi-flows.

\section{Definitions and Basic Properties}
Let $(X,d)$ be a compact metric space. 
Let $f:X\to X$ be a continuous map. 
Then $(X,f)$ is conventionally called a \emph{topological dynamical system} or just a
\emph{system}.
We shall denote by 
$\ZZ^+$ the set of all positive integers and by $\NN$ the set of all nonnegative integers, i.e.
$\NN=\ZZ^+\cup\{0\}$. 
For $n\in\ZZ^+$,  denote
$$\ZZ_n:=\{0,1,\cdots, n-1\}.$$
We fix a metric $D$ on $\cm(X,f)$ that induces the weak-$*$ topology.
Readers are referred to \cite{Walters} for preliminary knowledge about entropy and invariant measures.

\begin{definition}\label{gapshadow}
	Let $\sC=\{x_k\}
	_{k\in\ZZ^+}$ be a sequence in $X$.
	Let $\sss=\{m_k\}
	_{k\in\ZZ^+}$
	and $\sg=\{t_k\}_{k\in\ZZ^+}$
	be sequences of positive integers.
	For $\delta,\ep>0$ and $z\in X$, we say that $(\sC,\sss, \sg)$ is 
	\emph{$(\delta,\ep)$-traced} by $z$
	if 	for each $k\in\ZZ^+$, we have
	\begin{equation}\label{ndeltrace}
	\left|\left\{j\in\ZZ_{m_k}: d(f^{s_{k}+j}(z), f^j(x_k))>\ep\right\}\right|\le\delta
	m_k,
	\end{equation}
	where
	\begin{equation}\label{defsk}
	s_1=s_1(\sss,\sg):=0\text{ and }s_k=s_k(\sss,\sg):=\sum_{i=1}^{k-1}(m_i+t_i-1)\text{ for }k\ge 2.
	\end{equation}
\end{definition}

\begin{remark}
	Definition \ref{gapshadow} naturally extends to the case that
	$\sC,\sss,\sg$ are finite sequences, which allows us to define
	\emph{periodic} specification-like properties by asking that
	the tracing point $z$ is periodic.
\end{remark}

\begin{definition}\label{defapp}
	The system $(X,f)$ is said to have \emph{the approximate product property},
	if for every $\delta_1, \delta_2, \ep>0$, 
	there is $M=M(\delta_1,\delta_2,\ep)>0$
	such that for every $n> M$ and every sequence $\sC$ in $X$, there are a 
	sequence $\sg$ with $\max\sg\le1+\delta_1n$
	and $z\in X$ such that $(\sC,\{n\}^{\ZZ^+},\sg)$ is $(\delta_2,
	\ep)$-traced by $z$.
\end{definition}

The approximate product property was introduced by Pfister and Sullivan \cite{PfSu} to prove
large deviations for $\beta$-shifts.
Definition \ref{defapp} is equivalent to the definitions given in \cite{PfSu}
and \cite{Sunintent}.
Our proof of Theorem \ref{thmunierg} is based on the following fact related to entropy denseness, 
which is implicitly proved in \cite{Sunintent}.

\begin{proposition}[{cf. \cite{Sunintent}}]
	\label{entropydense}
	Let $(X,f)$ be a system with the approximate product property. 
	Then for every $\mu\in\cm(X,f)$,
	every $\eta,\ep,\beta>0$, there is a compact invariant subset $\Lambda=\Lambda(\mu,\eta,\ep,\beta)$ such that
	$D(\mu,\nu)<\eta$
	for every invariant measure $\nu$ supported on $\Lambda$
	and
	$h(\Lambda,f,\ep)<\beta$,
	where $h(\Lambda,f,\ep)$ is the topological entropy of $(\Lambda, f|_{\Lambda})$ calculated at the scale $\ep$.

\end{proposition}

It is the mistake $\delta_2>0$ allowed in \eqref{ndeltrace}
that introduces non-minimal examples of zero-entropy systems
with the approximate product property. So we suggest 
a stronger condition under which minimality can be guaranteed
by zero entropy and unique ergodicity.

\begin{definition}\label{defsapp}
	The system $(X,f)$ is said to have  the \emph{strict approximate product property},
	if for every $\delta, \ep>0$, 
	there is $M=M(\delta,\ep)>0$
	such that for every $n> M$ and every sequence $\sC$ in $X$,
	there are a 
	sequence $\sg$ with $\max\sg\le1+\delta n$
	and $z\in X$ such that $(\sC,\{n\}^{\ZZ^+},\sg)$ is $(0,
	\ep)$-traced by $z$.

\end{definition}

It is clear that the strict approximate product property is stronger than the approximate product property and weaker than the tempered gluing orbit property
introduced in \cite{Sunintent}. We remark that the difference between the strict
approximate product property and the tempered gluing orbit property is not in
the description of the gaps ($\delta n$ vs a tempered function), but
in the lengths of the orbit segments (equal lengths vs variable lengths)
that can be traced.

In \cite{Sunintent}, we have shown that minimality implies both unique ergodicity and zero topological entropy.

\begin{proposition}[{cf. \cite[Section 5.3]{Sunintent}}]\label{propminimal}
	Let $(X,f)$ be a system with the approximate product property. 
	If $(X,f)$ is minimal, then it is uniquely ergodic and has zero topological entropy.
\end{proposition}

Theorem \ref{thmsapp} is a combination of Theorem \ref{thmunierg}, Proposition
\ref{propminimal} and the following
proposition.

\begin{proposition}\label{uniergminisapp}
	Let $(X,f)$ be a system with the strict approximate product property. 
	If $(X,f)$ is uniquely ergodic then it is minimal.
\end{proposition}

%The idea of the proof of Proposition \ref{uniergminisapp} is similar to
%the proof of \cite[Theorem 4.1]{Sun19}.
%It is given below for completeness.

\begin{proof}
	The idea of the proof %of Proposition \ref{uniergminisapp} 
	is similar to the proof of 
	\cite[Theorem 4.1]{Sun19}.
	
	Assume that $(X,f)$ is not minimal and it has the strict approximate product property.
	We shall show that $(X,f)$ is not uniquely ergodic. 
	
	As the system is not minimal, there are $x,x'\in X$ and $\gamma>0$
	such that
	$$d\left(f^n(x),x'\right)\ge\gamma\text{ for every }n\in\NN.$$
	We fix $\ep\in(0,\frac\gamma3)$. Take a continuous function $\varphi:X\to\RR$
	such that
	\begin{align*}
	\varphi(y)=1&\text{ for every }y\in\overline{B(x',\ep)};\\
	%=\begin{cases}\end{cases}
	\varphi(y)=0&\text{ for every }y\notin B(x',2\ep);\\
	0<\varphi(y)<1&\text{ otherwise. }
	\end{align*}
	Then we have
	\begin{equation}\label{xsum}
	\lim_{n\to\infty}\frac1n\sum_{k=0}^{n-1}\varphi\left(f^k(x)\right)=0.
	\end{equation}

	Now we fix $\delta:=1$ and consider the sequence $\sC=\{x'\}^{\ZZ^+}$.
	As $(X,f)$ has the strict approximate product property, there are $m\in\ZZ^+$,
	a sequence $\sg:=\{t_k\}_{k\in\ZZ^+}$ with $\max\sg\le1+\delta m$ and $z\in X$ such that
	$\left(\sC,\{m\}^{\ZZ^+},\sg\right)$ is $(0,
	\ep)$-traced by $z$.
	Let $s_k:=s_k\left(\{m\}^{\ZZ^+},\sg\right)$ for each $k$ as in \eqref{defsk}.
	Then we must have
	$$f^{s_k}(z)\in \overline{B(x',\ep)}
	\text{ and }
	s_k\le(k-1)(m+1+\delta m-1)\le 2(k-1)m\text{ for each }k.$$
	This yields that
	\begin{equation}\label{zsum}
	\limsup_{n\to\infty}\frac1n\sum_{k=0}^{n-1}\varphi\left(f^k(z)\right)\ge
	\limsup_{k\to\infty}\frac1{s_k}\sum_{j=1}^{k-1}\varphi\left(f^{t_j}(z)\right)
	=\limsup_{k\to\infty}\frac {k-1}{s_k}\ge\frac1{2m}>0.
	\end{equation}
	Equations \eqref{xsum} and \eqref{zsum} imply that $\varphi$ does not converge
	pointwise to a constant.
	Hence $(X,f)$ is not uniquely ergodic.
\end{proof}

\begin{definition}\label{defpugo}
	We say that the system $(X,f)$  have the \emph{periodic exact specification property},
	if for every $\ep>0$, 
	there is a nonnegative integer $M=M(\ep)$
	such that for 
	every finite sequence $\sC=\{x_k\}_{k=1}^n$ in $X$ and every 
	finite sequence $\sss=\{m_k\}_{k=1}^n$ of positive integers,
	there 	is $z\in X$ such that $\left(\sC,\sss,\{M+1\}^{n}\right)$ is $(0,
	\ep)$-traced by $z$ and
	$$f^s(z)=z\text{ for }s:=\sum_{i=1}^{n}(m_k+M).$$

\end{definition}

We say that the system $(X,f)$ have the \emph{exact specification property} if the tracing
point $z$ in Definition \ref{defpugo} is not necessarily
%required to 
a periodic point.
It is clear that the periodic exact specification property implies the exact specification property, whose relations with other specification-like properties are illustrated
in Figure \ref{figrel}. It is shown in \cite{Blokh} and \cite{Buzzi} that
a continuous interval map $(I,f)$ has the periodic exact specification property if and
only if the map is topologically mixing.

\section{Proof of Theorem \ref{thmunierg}}
\subsection{Sufficiency}
In this subsection we prove the `if' part of  Theorem \ref{thmunierg}.
Let $(X,f)$ be a system with the approximate product property 
that is not uniquely ergodic. We shall show that
$(X,f)$ has positive topological entropy.

Let $\mu_1,\mu_2$ be two distinct ergodic measures for $(X,f)$.
Let 
$$\mu_3:=\frac{2\mu_1+\mu_2}{3},\mu_4:=\frac{\mu_1+2\mu_2}{3}
\text{ and }\eta:=\frac{D(\mu_1,\mu_2)}{7}.$$
Then
\begin{equation}\label{muijsep}
D(\mu_i,\mu_j)>2\eta\text{ for }1\le i<j\le 4.
\end{equation}
By Proposition \ref{entropydense},
there are compact invariant sets
$\Lambda_1$, $\Lambda_2$, $\Lambda_3$ and $\Lambda_4$
such that for each $i=1,2,3,4$ and for every invariant measure
$\nu$ supported on $\Lambda_i$, we have
$$D(\nu,\mu_i)<\eta.$$
%By \eqref{muijsep}, t
This implies that 
$$\Lambda_i\cap\Lambda_j=\emptyset\text{ for }1\le i<j\le 4.$$
Otherwise, if $\Lambda_i\cap\Lambda_j\ne\emptyset$ then
$\Lambda_i\cap\Lambda_j$ supports an invariant measure $\nu^*$ and hence
$$D(\mu_i,\mu_j)\le D(\mu_i,\nu^*)+D(\mu_j,\nu^*)< 2\eta,$$
which contradicts with \eqref{muijsep}.

By compactness, we have
$$d^*(\Lambda_i,\Lambda_j):=\min\left\{d(x,y):x\in\Lambda_i,y\in\Lambda_j\right\}>0
\text{ for }1\le i<j\le 4.$$
Denote
$$\gamma:=\frac14\min\left\{d^*(\Lambda_i,\Lambda_j):1\le i<j\le 4\right\}.$$
For each $i=1,2,3,4$, we fix a point $y_i\in\Lambda_i$.
Then for $1\le i<j\le 4$ and every $m,n\in\ZZ^+$, we have
\begin{equation}\label{yijsep}
d\left(f^m(y_i),f^n(y_j)\right)\ge 4\gamma.
\end{equation}

For each $\xi=\{\xi(k)\}_{k=1}^\infty\in\{1,2\}^{\ZZ^+}$, 
denote $\sC_\xi:=\{x_k(\xi)\}_{k=1}^\infty$
such that
\begin{equation}\label{xiconstruct}
x_{2k-1}(\xi)=y_{2\xi(k)-1}\text{ and }x_{2k}(\xi)=y_{2\xi(k)}.
\end{equation}
We fix 
$$\delta\in(0,\frac1{10})%, \delta_1:=3\delta, \delta_2:=\delta 
\text{ and }m>M(\delta,\delta,\gamma),$$
where $M(\delta,\delta,\gamma)$ is the constant obtained from the approximate product property as in Definition \ref{defapp}.
For each $\xi\in\{1,2\}^{\ZZ^+}$, we can find 
$\sg_\xi=\{t_k(\xi)\}_{k=1}^\infty$ with $\max\sg\le1+\delta m$ 
and $z_\xi\in X$ such that
$\left(\sC_\xi,\{m\}^{\ZZ^+},\sg_\xi\right)$ is $(\delta,\gamma)$-traced by $z_\xi$.
Denote $s_k(\xi):=s_k\left(\{m\}^{\ZZ^+},\sg_\xi\right)$ as in \eqref{defsk}
for each $k$. 

\begin{lemma}\label{lemsep}
	If there is $n\in\{1,\cdots, N\}$ such that
	$\xi(n)\ne \xi'(n)$, then $z_\xi$ and $z_{\xi'}$
	are $((1+\delta)nm,\gamma)$-separated.
\end{lemma}

\begin{proof}
	The proof splits into two cases.

	\textbf{\flushleft Case 1.} Suppose that
	$$r:=\left|s_n(\xi)-s_n(\xi')\right|\le 4\delta m.$$
	We may assume that $s_n(\xi)>s_n(\xi')$.
	By the tracing property, there are $A, A'\in\ZZ_m$ such that
	$$|A|,|A'|\ge(1-\delta)m,$$
	$$d\left(f^{s_{n}(\xi)+j}(z_\xi), f^j\left(x_n(\xi)\right)\right)
	\le\gamma\text{ for each }j\in A$$
	and
	$$d\left(f^{s_{n}(\xi')+j}(z_{\xi'}), f^j\left(x_n(\xi')\right)\right)
	\le\gamma\text{ for each }j\in A'.$$
	Then we must have
	$$\left|(r+A)\cap A'\right|\ge|r+A|+|A'|-|r+\ZZ_m|\ge(1-6\delta)m>0.$$
	For $l\in(r+A)\cap A'$, by \eqref{yijsep}, we have
	\begin{align*}
	{}&{}d\left(f^{s_{n}(\xi')+l}(z_{\xi}), f^{s_{n}(\xi')+l}(z_{\xi'})\right)
	\\\ge{}&{}d\left(f^{l-r}\left(x_n(\xi)\right), f^l\left(x_n(\xi')\right)\right) 
	%\\&{}
	-d\left(f^{s_{n}(\xi)+(l-r)}(z_{\xi}), f^{l-r}\left(x_n(\xi)\right)\right)
	\\&{}
	-d\left(f^{s_{n}(\xi')+l}(z_{\xi'}), f^l\left(x_n(\xi')\right)\right)
	\\\ge{}&{}4\gamma-\gamma-\gamma>\gamma.
	\end{align*}
	Moreover, we have 
	$$s_{n}(\xi')+l\le\sum_{k=1}^{n-1}(m+t_{k}(\xi')-1)+m\le (1+\delta)nm.$$
	So $z_\xi$ and $z_{\xi'}$ are $((1+\delta)nm,\gamma)$-separated.
	
	\textbf{\flushleft Case 2.}
	Suppose that $\tau$ is the smallest positive integer such that
	\begin{equation*}\label{smalltau}
	\tau\le n\text{ and }\left|s_\tau(\xi)-s_\tau(\xi')\right|>4\delta m.
	\end{equation*}
	Then we have $\tau>1$ and 
	$$\left|s_{\tau-1}(\xi)-s_{\tau-1}(\xi')\right|\le 4\delta m.$$
	We may assume that $s_\tau(\xi)>s_\tau(\xi')$.
	Then we have
	\begin{align}
	s_\tau(\xi)-s_\tau(\xi')
	&=s_{\tau-1}(\xi)+(m+t_{\tau-1}(\xi)-1)-s_{\tau-1}(\xi')-(m+t_{\tau-1}(\xi')-1)\notag
	\\&\le s_{\tau-1}(\xi)-s_{\tau-1}(\xi')+(1+\delta m)-1\notag
	\\\label{staudiff}&\le 5\delta m.
	\end{align}
	This yields that
	$$r:=s_{\tau+1}(\xi')-s_\tau(\xi)
	=s_{\tau}(\xi')+(m+t_{\tau}(\xi')-1)-s_{\tau}(\xi)
	\le (1-3\delta) m$$
	and by \eqref{staudiff} we have
	$$r\ge m-(s_\tau(\xi)-s_\tau(\xi'))\ge(1-5\delta)m>0.$$
	By the tracing property, there are $A, A'\in\ZZ_m$ such that
	$$|A|,|A'|\ge(1-\delta)m,$$
	$$d\left(f^{s_{\tau}(\xi)+j}(z_\xi), f^j\left(x_\tau(\xi)\right)\right)
	\le\gamma\text{ for each }j\in A$$
	and
	$$d\left(f^{s_{\tau+1}(\xi')+j}(z_{\xi'}), f^j\left(x_{\tau+1}(\xi')\right)\right)
	\le\gamma\text{ for each }j\in A'.$$
	Then we must have
	$$\left|(r+A)\cap A'\right|\ge|r+A|+|A'|-|r+\ZZ_m|\ge \delta m>0.$$
	Note that by our construction \eqref{xiconstruct}, 
	we must have
	$x_\tau(\xi)\ne x_{\tau+1}(\xi')$. 
	For $l\in(r+A)\cap A'$, by \eqref{yijsep}, we have
	\begin{align*}
	{}&{}d\left(f^{s_{\tau}(\xi)+l}(z_{\xi}), f^{s_{\tau}(\xi)+l}(z_{\xi'})\right)
	\\\ge{}&{}d\left(f^{l}\left(x_\tau(\xi)\right), f^{l-r}\left(x_{\tau+1}(\xi')\right)\right) 
	-d\left(f^{s_{\tau}(\xi)+l}(z_{\xi}), f^{l}\left(x_\tau(\xi)\right)\right)
	\\&{}
	-d\left(f^{t_{\tau+1}(\xi')+(l-r)}(z_{\xi'}), f^{l-r}\left(x_{\tau+1}(\xi')\right)\right)
	\\\ge{}&{}4\gamma-\gamma-\gamma>\gamma.
	\end{align*}
	Moreover, we have 
	$$s_{\tau}(\xi)+l\le\sum_{k=1}^{\tau-1}(m+t_{k}(\xi')-1)+m\le (1+\delta)nm.$$
	So $z_\xi$ and $z_{\xi'}$ are $((1+\delta)nm,\gamma)$-separated.
\end{proof}

By Lemma \ref{lemsep}, for each $n$,
there is a $((1+\delta)nm,\gamma)$-separated set whose cardinality is $2^n$.
This yields that
$$h(f)\ge\limsup_{n\to\infty}\frac{\ln 2^n}{(1+\delta)nm}=\frac{\ln2}{(1+\delta)m}>0.$$

\subsection{Necessity}

In this subsection we prove the `only if' part of  Theorem \ref{thmunierg}.
Suppose that $(X,f)$ is a system with the approximate product property and positive topological entropy
$h(f)>0$. We shall show that such a system has more than one ergodic measures.
Compared with the results in \cite{Sunintent}, in what follows we do not assume that
the system is asymptotically entropy expansive.

By Proposition \ref{entropydense}, for each $k$, there is a compact invariant set
$\Lambda_k$ such that
\begin{equation}\label{entest}
h(\Lambda_k,f,\frac1k):=\limsup_{n\to\infty}\frac{\ln s(\Lambda_k, n, \frac1k)}{n}<\frac1k,
\end{equation}
where $s(\Lambda_k, n, \frac1k)$ denotes the maximal cardinality of 
$(n,\frac1k)$-separated subsets of $\Lambda_k$. 
Denote
$$\Gamma_k:=\bigcap_{j=1}^k \Lambda_j \text{ for each }k.$$
Then for each $k$, $\Gamma_k$ is also compact and $f$-invariant.

There are two cases to consider:
\begin{enumerate}
	\item Suppose that there is $k$ such that $\Gamma_k\cap\Lambda_{k+1}=\emptyset$.
	Then there are two distinct ergodic measures that are supported on $\Gamma_k$ and
	$\Lambda_{k+1}$, respectively.
	Then $(X,f)$ is not uniquely ergodic.
	\item Suppose that $\Gamma_k\cap\Lambda_{k+1}\ne\emptyset$ for all $k$.
	In this case we have a nonempty invariant compact set
	$$\Gamma:=\bigcap_{k=1}^\infty \Gamma_k=\bigcap_{k=1}^\infty \Lambda_k.$$
	By \eqref{entest}, we have
	$$h(\Gamma,f,\frac1k)\le h(\Lambda_k,f,\frac1k)<\frac1k\text{ for all }k.$$
	This implies that 
	$$h(\Gamma,f)=\lim_{k\to\infty}h(\Gamma,f,\frac1k)=0.$$ 
	and hence $\Gamma$ supports an ergodic
	measure of zero entropy. However, as $h(f)>0$, the system $(X,f)$ must have ergodic
	measures of positive entropy. This completes the proof.
\end{enumerate}

\section{Examples}\label{seceg}

\begin{example}\label{exintmap}
	According to \cite{MS}, there is a class of zero-entropy $C^\infty$ interval maps 
	such that each map in the class
	has periodic points of period $2^n$ for any $n\in\ZZ^+$ 
	and is chaotic in the sense
	of Li-Yorke. Theorem \ref{thmunierg} implies that these maps do not have 
	the approximate product property.
\end{example}

\begin{example}\label{exaper}
	In \cite{HK}, a minimal subshift is constructed
	as the first example of strictly ergodic system with positive
	topological entropy, which gives a negative answer to Parry's question. 
	In \cite{BC}, it is shown that 
	strictly ergodic homeomorphisms with positive topological entropies
	can be constructed on 
	every compact
	manifold whose dimension is at least 2 and that carries a strictly ergodic homeomorphism.
	By Theorem \ref{thmunierg}, these systems do not have the approximate product
	property.
\end{example}

\begin{example}
	\label{exnonmini}
	By \cite[Theorem 7.1]{Kwiet} and \cite[Example 9.6]{Sunintent}, there is 
	$X_1\subset\{0,1\}^{\NN}$ such that
	$$\frac1n\max\{|\{m\le k<m+n:w_k=1\}|:m\in\NN\}$$
	converges uniformly to zero for every $\{w_k\}_{k\in\NN}\in X_1$.
	Consider the subshift $\sigma$ on $X_1$.
	Then for every sequence $\sC$ in $X_1$, every $\delta_2,\ep>0$,
	there is $M\in\ZZ^+$ such that for every $n>M$,
	$\left(\sC,\{n\}^{\ZZ^+},\{1\}^{\ZZ^+}\right)$ is $(\delta_2,
	\ep)$-traced by the fixed point $\{0\}^{\NN}$. %\in X_1$.
	The subshift $(X_1,\sigma)$ has the approximate product property and zero topological
	entropy. It is topologically mixing but not minimal.
	We remark that by Theorem \ref{thmsapp}, this subshift does not have the strict approximate
	product property.
	
	Such subshifts can be modified to obtain a non-transitive system with the approximate	product property. See \cite[Example  9.7]{Sunintent}.
\end{example}

\begin{example}\label{herman}
	In \cite{Herman}, Herman
	constructed a family of $C^\infty$ diffeomorphisms
	$\cF=\{F_\alpha:\alpha\in\mathbb{T}^1\}$ on $X=\mathbb{T}^1\times\mathrm{SL}(2,\mathbb{R})/\Gamma$, where
	$\mathbb{T}^1=[0,1]/\sim$ is the unit circle and
	$\Gamma$ is a cocompact discrete subgroup of $\mathrm{SL}(2,\mathbb{R})$. For each $\alpha\in\mathbb{T}^1$,
	$$F_\alpha(\theta,g\Gamma)=(R_\alpha(\theta),A_\theta(g\Gamma))$$ is a skew product, where $R_\alpha(\theta)=\theta+\alpha$
	is the rotation on $\mathbb{T}^1$, 
	$$A_\theta(g\Gamma)=\left(\begin{array}{r r}
	\cos 2\pi\theta & -\sin 2\pi\theta \\ \sin 2\pi\theta & \cos
	2\pi\theta
	\end{array} \right)
	\left(\begin{array}{r r}\lambda & 0 \cr 0 &
	1/\lambda\end{array}\right)g\Gamma\text{ for each }g\Gamma\in\mathrm{SL}(2,\mathbb{R})/\Gamma$$
	and $\lambda>1$ is a fixed real number.
	Herman showed in \cite{Herman} that 
	$h(F_\alpha)>0$ for every $\alpha\in\mathbb{T}^1$ and
	there is dense $G_\delta$ subset $W\in\mathbb{T}^1$ such that
	$F_\alpha$ is minimal for every $\alpha\in W$.
	
	By Proposition \ref{propminimal},  the approximate product
	property does not hold for $F_\al$ as long as  $\alpha\in
	W$.  We doubt if it holds for any element $F_\alpha$ in the family $\cF$.

\end{example}

\section{Periodic Points and Unique Ergodicity }
\subsection{Systems with Fixed Points}

We perceive that Example \ref{exnonmini} exhibits a more general phenomenon,
which may be regarded as a flaw that appears naturally with the way
we define the approximate product property, where we allow mistakes in
the tracing property \eqref{ndeltrace}.

\begin{proposition}\label{propfix}
	Let $(X,f)$ be a system with a fixed point $p\in X$. Then
	$(X,f)$ has the approximate product property and zero topological entropy 
	if and only if $(X,f)$ is uniquely
	ergodic, i.e. The Dirac measure on $\{p\}$ is the unique ergodic measure
	for $(X,f)$.
\end{proposition}

\begin{proof}
	The 'only if' part is a corollary of Theorem \ref{thmunierg}. Now we assume
	that $\delta_p$ is the unique ergodic measure for $(X,f)$. By the Variational
	Principle, we must have
	$h(f)=h_{\delta_p}(f)=0$. We need to show that $(X,f)$ has the approximate
	product property. In fact, every sequence can be traced by the orbit of the
	fixed point $p$. This is analogous to the situation in Example \ref{exnonmini}.
	
	Let $\ep>0$. There is a continuous function $\varphi:X\to\RR$
	such that
	\begin{align*}
	\varphi(x)=0&\text{ for }x=p;\\
	%=\begin{cases}\end{cases}
	\varphi(x)=1&\text{ for every }x\notin B(p,\ep);\\
	0<\varphi(x)<1&\text{ otherwise. }
	\end{align*}
	As $(X,f)$ is uniquely ergodic, we have that
	$$\frac1n\sum_{j=0}^{n-1}\varphi(f^j(x))\to \frac1n\sum_{j=0}^{n-1}\varphi(f^j(p))
	=0\text{ uniformly}.$$
	Then for every $\delta_2>0$, there is $M\in\ZZ^+$ such that
	$$\frac1n\sum_{j=0}^{n-1}\varphi(f^j(x))<\delta_2\text{ for every $n>M$ and
		every $x\in X$}.$$
	This implies that
	$$\left|\left\{j\in\ZZ_n: d(f^j(x),p)>\ep\right\}\right|<\delta_2n.$$
	Then for every sequence $\sC$ in $X$, we have
	that $(\sC,\{n\}^{\ZZ^+},\{1\}^{\ZZ^+})$ can be
	$(\delta_2,\ep)$-traced by the fixed point $p$.
	Hence $(X,f)$ has the approximate product property.
\end{proof}

\begin{remark}
	The proof of Proposition \ref{propfix} in fact shows that $(X,f)$ has the almost specification property
	if it has a unique ergodic measure supported on a fixed point.
\end{remark}

Theorem \ref{thmper} can be verified as a corollary of Proposition \ref{propfix}
and the following fact, which holds generally for systems with the approximate product property.

\begin{proposition}
	Suppose that there is a positive integer $N$ such that $(X,f^N)$ is a system
	with the approximate product property. Then so is $(X,f)$.
\end{proposition}

\begin{proof}
	Suppose that we are given $\delta_1,\delta_2,\ep>0$. By continuity, there
	is $\gamma>0$ such that
	\begin{equation}\label{dfncont}
	d(f^j(x),f^j(y))<\ep\text{ for every $j\in\ZZ_N$, whenever }d(x,y)<\gamma.
	\end{equation}
	As $(X,f^N)$ has the approximate product property, there is $M=M(\frac{\delta_1}2,
	\frac{\delta_2}2,\gamma)$
	as  in Definition \ref{defapp} such that for every $n> M$ and every sequence
	$\sC$ in $X$, there are an %$(n,\delta_1)$-spaced
	sequence $\sg=\{t_k\}_{k\in\ZZ^+}$ with $\max\sg\le1+\frac{\delta_1n}2$
	and $z\in X$ such that $\left(\sC,\{n\}^{\ZZ^+},\sg\right)$ is $(\frac{\delta_2}2,
	\ep,\gamma)$-traced by $z$ under $f^N$. 
	By \eqref{dfncont}, this implies that $\left(\sC,\{nN\}^{\ZZ^+},\{1+N(t_k-1)\}_{k\in\ZZ^+}\right)$ 
	is $(\frac{\delta_2}2,
	\ep)$-traced by $z$ under $f$. 
	
	Let 
	\begin{equation}\label{estoft}
	T>\max\left\{M,1+\frac2{\delta_1},2\right\}.
	\end{equation}
	For every $n>TN$, we can write $n=mN+l$ such that
	$m\ge T$ and $0\le l<N$.
	Suppose that we have
	$\sg=\{t_k\}_{k\in\ZZ^+}$ with $\max\sg\le1+\frac{\delta_1(m+1)}2$
	and 
	$\left(\sC,\{(m+1)N\}^{\ZZ^+},\{1+N(t_k-1)\}_{k\in\ZZ^+}\right)$ 
	is $(\frac{\delta_2}2,
	\ep)$-traced by $z$ under $f$.
	By \eqref{estoft}, we have
	$$1+N(t_k-1)+(N-l)\le1+\frac{\delta_1(m+1)N}2+N<1+\delta_1n$$
	for each $k$ and
	$$\frac{\delta_2}{2}(m+1)N<\delta_2 mN\le\delta_2n.$$
	These bounds guarantee that 
	the gap $\sg'=\{1+N(t_k-1)+(N-l)\}_{k\in\ZZ^+}$ satisfies
	$\max\sg'\le1+\delta_1n$ and
	$\left(\sC,\{n\}^{\ZZ^+},\sg'\right)$ 
	is $(\delta_2,
	\ep)$-traced by $z$ under $f$.
	Hence $(X,f)$ also has the approximate product property.
\end{proof}

\subsection{Interval Maps}

Let $I$ be a closed interval. We know that every continuous interval map 
%$f:I\to I$
must have a fixed point and hence is not minimal. By Theorem \ref{thmsapp}, 
there is no continuous
interval map that has the strict approximate product property and zero topological
entropy.

Now we would like to prove Theorem \ref{zentappinterval}.
Suppose that $f:I\to I$ is a continuous map on the interval $I$.
As $f$ has a fixed point in $I$, the `if' part of Theorem
\ref{zentappinterval} is a corollary of Proposition \ref{propfix}.
%Then $f$ must have a fixed point in $I$. Hence the `if' part of Theorem
%\ref{zentappinterval} is a corollary of Proposition \ref{propfix}.

Suppose that $(I,f)$ has the approximate product property and zero topological entropy.
By Theorem \ref{thmper}, $(I,f)$  has a unique ergodic measure which is supported
on a unique
fixed point $p\in I$ and  $(I,f)$ has no other periodic points. 
The following fact completes the proof of Theorem \ref{zentappinterval}.

\begin{proposition}\label{propattract}
	Let $p$ be a fixed point of a continuous interval map $(I,f)$. Suppose that $(I,f)$
	has no other periodic points. Then $p$ is attracting.
\end{proposition}

We remark that in general, as exhibited in Example \ref{exnonmini}, the unique fixed point
of a uniquely ergodic system is not necessarily attracting.
We shall give a proof of Proposition \ref{propattract} for completeness.
It is well-known that a continuous map $f:\RR\to\RR$ has a fixed point in
an interval $J$ if 
$f(J)\supset J$. 
We shall use this fact and the Darboux property that
$f([a,b])\supset[f(a),f(b)]$ constantly without reference.

Suppose that the assumptions of Proposition \ref{propattract} holds. Let
$x\in I$ and $x\ne p$. 
To prove
Proposition \ref{propattract} we need to show that $f^n(x)\to p$
as $n\to\infty$.

\begin{lemma}\label{fnxgex}
	If $x<p$ then 
	$f^n(x)>x$ for every $n\in\ZZ^+$.
	If $x>p$ then $f^n(x)<x$ for every $n\in\ZZ^+$.
\end{lemma}

\begin{proof}
	We give the proof for the case that $x<p$. The proof of the other case is analogous.
	As $x$ is not a periodic point,
	we have $f^n(x)\ne x$ for all $n$.
	Suppose that there is $n\in\ZZ^+$ such that $f^n(x)<x$. 
	We shall show that this leads to the existence of another periodic point of $f$.
	
	There are two cases
	to consider:
	\begin{enumerate}
		\item Suppose that there is $m\in\ZZ^+$ such that 
		$f^{mn}(x)<x<f^{(m+1)n}(x)$.
		
		If $f^{mn}(x)<f^n(x)$, then we have
		$$f^{mn}([f^n(x),x])\supset [f^{mn}(x),f^{(m+1)n}(x)]\supset[f^n(x),x],$$
		hence there is an $mn$-periodic point in $[f^n(x),x]$ that is different from
		$p$.
		Otherwise it holds that $f^{n}(x)\le f^{mn}(x)$. In this case
		we have
		$$f^{n}([f^{mn}(x),x])\supset [f^{n}(x),f^{(m+1)n}(x)]\supset[f^{mn}(x),x],$$
		hence there is an $n$-periodic point in
		$[f^{mn}(x),x]$ that is different from $p$.
		
		\item Suppose that $f^{mn}(x)<x$ for every $m\in\ZZ^+$.
		Let $$a:=\inf\{f^{mn}(x):m\in\ZZ^+\}\in I.$$
		Then we must have $f^n(a)\ge a$. But $f^n(x)<x$. By the Intermediate Value
		Theorem, there is a fixed point of $f$ in the interval $[a,x]$ that is
		different from $p$.

	\end{enumerate}
\end{proof}

\begin{proof}[Proof of Proposition \ref{propattract}]
	Let $x\in I$ and $x\ne p$. 
	If there is $f^N(x)=p$ for some $N$ then $f^n(x)=p$ for all $n>N$.
	
	Suppose that $f^n(x)\ne p$ for all $n$. There is a decomposition $\NN=\cN_1\cup \cN_2$ such
	that
	$$\begin{cases}f^n(x)<p, &\text{ for }n\in \cN_1;\\
	f^n(x)>p, &\text{ for }n\in \cN_2.
	\end{cases}$$
	The proof splits into two cases:
	\begin{enumerate}
		\item Suppose that $\cN_2$ is finite. Then there is $N$ such that $n\in\cN_1$
		for all $n\ge N$. %Then 
		By Lemma \ref{fnxgex}, $\{f^n(x)\}_{n=N}^\infty$ is an increasing sequence
		with an upper bound $p$. Hence it converges to some $q\in I$. Then
		$$f(q)=f\left(\lim_{n\to\infty}f^n(x)\right)=\lim_{n\to\infty}f^{n+1}(x)=q.$$
		As $p$ is the unique fixed point, we must have $q=p$ and hence $f^n(x)\to p$
		as $n\to\infty$. Analogous argument works for case that $\cN_1$ is finite.
		\item Suppose that both $\cN_1$ and $\cN_2$ are infinite. 
		Arrange the elements of $\cN_1$ in a strictly increasing sequence $\{n_k\}_{k=1}^\infty$
		and the elements of $\cN_2$ in a strictly increasing sequence $\{n_k'\}_{k=1}^\infty$.
		By Lemma \ref{fnxgex},
		$\{f^{n_k}(x)\}_{k=1}^\infty$ is an increasing sequence with an upper bound
		$p$ and $\{f^{n_k'}(x)\}_{k=1}^\infty$ is an decreasing sequence with a lower bound
		$p$. Then
		there are $q_1,q_2\in I$ such that
		$$\lim_{n_k\to\infty} f^{n_k}(x)=q_1\text{ and }\lim_{n_k'\to\infty} f^{n_k'}(x)=q_2.$$
		We see that $q_1,q_2$ are exactly all the subsequential limits of the sequence
		$\{f^n(x)\}_{n=1}^\infty$.
		So we must have $f(\{q_1,q_2\})\subset\{q_1,q_2\}$.
		There are three sub-cases:
		\begin{enumerate}
			\item If $q_1=q_2=p$ then we have $f^n(x)\to p$ as $n\to\infty$.
			\item If $q_1\ne p$ and $q_2\ne p$. 
			As $p$ is the unique fixed point, we
			have $f(q_1)\ne q_1$ and $f(q_2)\ne q_2$.
			This yields that $f(q_1)=q_2$ and $f(q_2)=q_1$.
			Then $q_1,q_2$ are 2-periodic points, which is a contraction.
			\item Without loss of generality, we may assume that $q_1\ne p$ and $q_2=p$.
			Then $f(q_1)=q_2=p$. As $f$ is continuous, there is $\delta>0$ such that
			\begin{equation}\label{fypclose}
			|f(y)-p|<|q_1-p|\text{ for every }y\in B(p,\delta).
			\end{equation}
			As $f^{n_k'}(x)\to p$, there is $N\in\cN_2$ such that $p<f^N(x)<p+\delta$.
			By \eqref{fypclose}, we have $|f^{N+1}(x)-p|<|q_1-p|$. But
			$f^{n_k}(x)\le q_1<p$ for every $n_k\in\cN_1$. This implies that
			$N+1\in\cN_2$ and hence by
			Lemma \ref{fnxgex} we have $$p< f^{N+1}(x)<f^N(x)<p+\delta.$$ %Then 
			Analogous argument
			shows that
			$$p< f^{N+2}(x)<f^{N+1}(x)<p+\delta.$$ 
			Hence by induction we have
			$n\in\cN_2$ for every $n\ge N$, which contradicts with the assumption that
			$\cN_1$ is infinite.
		\end{enumerate}
		
		%Suppose that $q_1\ne p$. As $p$ is the unique fixed point, we must have $f(q_1)=q_2$
	\end{enumerate}
\end{proof}

\section*{Acknowledgments}
The author is supported by National Natural Science Foundation of China (No. 11571387)
and CUFE Young Elite Teacher Project (No. QYP1902). The author would like to
thank Xueting Tian, Jian Li and the anonymous referees for helpful comments.


\begin{thebibliography}{99}

\bibitem{ADR}
Alsed\`a, LL., Del R\'io, M. A., Rodr\'iguez, J. A.: 
Transitivity and dense periodicity for graph maps. 
\emph{Journal of Difference Equations and Applications}, \textbf{9}(6), 577-598 (2003)

\bibitem{BC} B\'eguin, F., Crovisier, S., Le Roux, F.: 
Construction of curious
	minimal uniquely ergodic homeomorphisms on manifolds: 
	the Denjoy-Rees technique. 
\emph{Ann. Sci. \'Ecole Norm. Sup.}, \textbf{40}(4), 251--308 (2007)


\bibitem{Blokh}
Blokh A.M.: 
Decomposition of dynamical systems on an interval. 
\emph{Russ. Math. Surv.}, \textbf{38}, 133--134 (1983)



\bibitem{BTV} 
 Bomfim, T., Torres, M. J., Varandas, P.:
Topological features of flows with the
	reparametrized gluing orbit property. 
\emph{Journal of Differential Equations}, \textbf{262}(8), 4292--4313 (2017)


\bibitem{Bowen} Bowen, R.: 
Periodic points and measures for Axiom A diffeomorphisms. 
\emph{Trans. Amer. Math. Soc.}, \textbf{154}, 377--397 (1971)

\bibitem{Buzzi}  Buzzi, J.:
Specification on the interval.
\emph{Trans. Amer. Math. Soc.}, \textbf{349}(7), 2737--2754 (1997)

\bibitem{DGS} Denker, M., Grillenberger, C., Sigmund, K.:
Ergodic theory on compact spaces. 
\emph{Lecture Notes in Mathematics}, 
Vol. 527, Springer-Verlag, Berlin, 1976




\bibitem{GSW}
 Guan, L., Sun, P., Wu, W.:
Measures of intermediate entropies and homogeneous dynamics. 
\emph{Nonlinearity}, \textbf{30}, 3349--3361 (2017)



\bibitem{HK} Hahn, F., Katznelson, Y.:
On the entropy of uniquely ergodic transformations.
\emph{Trans. Amer. Math. Soc.}, \textbf{126}, 335--360 (1967)

\bibitem{Herman} Herman, M.:
Construction d'un diff\'{e}omorphisme minimal d'entropie
	topologique non nulle.
\emph{Ergodic Theory and Dynamical Systems}, \textbf{1}, 65--76 (1981)

\bibitem{Ka80} Katok, A.:
Lyapunov exponents, entropy and periodic orbits for diffeomorphisms.
\emph{Publ. Math. I.H.E.S.}, \textbf{51}, 137--173 (1980)



\bibitem{KKK}
 Konieczny, J., Kupsa, M., Kwietniak, D.:
Arcwise connectedness of the set of ergodic measures of hereditary shifts. 
\emph{Proceedings of the American Mathematical Society}, \textbf{146}(8), 3425-3438 (2018)

\bibitem{Kwiet}
Kwietniak, D.:
Topological entropy 
	and distributional chaos in hereditary shifts with applications 
	to spacing shifts and beta shifts. 
\emph{Discrete and Continuous Dynamical Systems - A}, \textbf{33}(6), 2451-2467 (2013)

\bibitem{KLO}
Kwietniak, D., Lacka, M., Oprocha, P.:
A panorama of specification-like properties and their consequences.
\emph{Contemporary Mathematics}, \textbf{669}, 155--186 (2016)


\bibitem{LOS}
Lindenstrauss, J., Olsen,  G., Sternfeld,  Y.:
The Poulsen simplex. 
\emph{Ann. Inst. Fourier (Grenoble)},
\textbf{28}(1), 91--114 (1978)

\bibitem{MS}
 Misiurewicz, M., Sm\'ital, J.: 
Smooth chaotic maps with zero topological entropy. 
\emph{Ergodic Theory and Dynamical Systems}, \textbf{8}(3), 421--424 (1988)



\bibitem{PfSu}
 Pfister, C. -E., Sullivan,  W. G.:
Large deviations estimates for dynamical systems without the
	specification property. Application to the $\beta$-shifts. 
\emph{Nonlinearity}, \textbf{18}, 237--261 (2005)

\bibitem{Phelps}
Phelps, R. R.:
Lectures on Choquet's theorem. second ed.,
\emph{Lecture Notes in Mathematics}, Vol.
1757, Springer-Verlag, Berlin, 2001

\bibitem{QS}
Quas, A., Soo, T.:
Ergodic universality of some topological dynamical systems,
\emph{Transactions of the American Mathematical Society}, \textbf{368}(6), 4137--4170 (2016)


\bibitem{Sun09}
Sun, P.:
Zero-entropy invariant measures for skew product diffeomorphisms.
\emph{Ergodic Theory and Dynamical Systems}, \textbf{30}, 923--930 (2010)

\bibitem{Sun10}
Sun, P.:
Measures of intermediate entropies for skew product diffeomorphisms.
\emph{Discrete and Continuous Dynamical Systems - A}, \textbf{27}(3), 1219--1231 (2010)

\bibitem{Sun12}
Sun, P.:
Density of metric entropies for linear toral automorphisms.
\emph{Dynamical
Systems}, \textbf{27}(2), 197--204 (2012)

\bibitem{Sun19}
Sun, P.:
Minimality and gluing orbit property.
\emph{Discrete and Continuous Dynamical Systems - A},
\textbf{39}(7), 4041-4056 (2019)


\bibitem{Sun18flow}
Sun, P.:
On the entropy of flows with
	reparametrized gluing orbit property.
\emph{Acta Mathematica Scientia}, \textbf{40B}(3), 855-862 (2020)

\bibitem{Sun20}
Sun, P.:
Denseness of intermediate pressures for systems with the Climenhaga-Thompson structure,
\emph{Journal of Mathematical Analysis and Applications}, 
\textbf{487}(2), 124027 (2020)

\bibitem{Sunze}
Sun, P.:
Zero-entropy dynamical systems with the gluing orbit property.
%arXiv:1810.08980, 
\emph{Advances in Mathematics}, to appear


\bibitem{Sunintent}
Sun, P.:
Ergodic measures of intermediate entropies for dynamical systems with approximate
	product property.
arXiv:1906.09862, 2019

\bibitem{Sunes}
Sun, P.:
Equilibrium states of intermediate entropies.
arXiv:2006.06358, 2020


\bibitem{Ures}
Ures, R.:
Intrinsic ergodicity of partially hyperbolic diffeomorphisms with a hyperbolic linear part. 
\emph{Proceedings of the American Mathematical Society}, \textbf{140}(6), 1973-1985 (2012)

\bibitem{Walters}
 Walters, P.:
An Introduction to Ergodic Theory,
Springer-Verlag, New York, 1982









\end{thebibliography}
\end{document}